\documentclass{amsart}
\usepackage{enumerate} 
\usepackage{amsfonts}
\usepackage{tikz}
\usepackage[numbers, sort&compress]{natbib}
\usepackage{amsmath,amsthm,amsfonts,amssymb,fancybox,float}
\usepackage{framed}
\usepackage{amsmath}
\usepackage{color,soul}
\usepackage{hyperref}

\newtheorem{theorem}{Theorem}[section]

\newtheorem{proposition}[theorem]{Proposition}

\theoremstyle{definition}
\newtheorem{remark}[theorem]{Remark}

\newtheorem{example}[theorem]{Example}
\newtheorem{assumption}[theorem]{Assumption}

\numberwithin{equation}{section}

\newcommand{\comment}[1]{}
\def\fddto{\stackrel{\rm f.d.d.}{\Longrightarrow}}
\newcommand{\ind}{{\bf 1}}

\def\inddd#1{{\ind}_{\left\{#1\right\}}}

\newcommand{\proba}{\mathbb P}
\renewcommand{\P}{\mathbb P}
\newcommand{\esp}{{\mathbb E}}

\newcommand{\cov}{{\rm{Cov}}}
\newcommand{\var}{{\rm{Var}}}

\newcommand{\eqnh}{\begin{eqnarray*}}
\newcommand{\eqne}{\end{eqnarray*}}
\newcommand{\eqnhn}{\begin{eqnarray}}
\newcommand{\eqnen}{\end{eqnarray}}
\newcommand{\equh}{\begin{equation}}
\newcommand{\eque}{\end{equation}}

\def\summ#1#2#3{\sum_{#1 = #2}^{#3}}

\newcommand{\eqd}{\stackrel{d}{=}}

\def\topp#1{^{(#1)}}

\def\nn#1{{\left\|#1\right\|}}

\def\abs#1{\left|#1\right|}

\def\ccbb#1{\left\{#1\right\}}
\def\sccbb#1{\{#1\}}
\def\pp#1{\left(#1\right)}
\def\spp#1{(#1)}
\def\bb#1{\left[#1\right]}

\def\mmid{\;\middle\vert\;}
\def\mmod{\ {\rm mod }\ }
\def\ip#1{\left\langle#1\right\rangle}

\def\aa#1{\left\langle #1\right\rangle}

\def\vv#1{{\boldsymbol #1}}

\def\vvs{{\boldsymbol s}}
\def\vvt{{\boldsymbol t}}

\def\qmwith{\quad\mbox{ with }\quad}
\def\mfa{\mbox{ for all }}

\def\mmas{\mbox{ as }}

\def\wt#1{\widetilde{#1}}
\def\wb#1{\overline{#1}}

\def\R{{\mathbb R}}
\def\Rd{{\mathbb R^d}}
\def\N{{\mathbb N}}

\def\B{{\mathbb B}}
\def\C{{\mathbb C}}
\def\D{{\mathbb D}}
\def\E{{\mathbb E}}
\def\H{{\mathbb H}}

\def\calA{\mathcal A}

\def\calE{\mathcal E}
\def\calF{\mathcal F}

\def\calM{\mathcal M}
\def\calN{\mathcal N}

\def\d{\mathsf d}
\def\MM{\mathsf M}
\def\Hn{\mathbb H^n}
\def\Rn{\mathbb R^n}
\def\Sn{\mathbb S^n}

\def\vvA{{\boldsymbol A}}
\def\vvN{{\boldsymbol N}}

\def\eab{\eta_{\alpha,\beta}}
\def\xab{\xi_{\alpha,\beta}}
\def\MpE{{\mathbb M_p(E)}}
\def\Mab{M_{\alpha,\beta}}

\title[Stable processes indexed by metric spaces]{Stable processes with stationary increments parameterized by metric spaces}
\author{Zuopeng Fu}\address{Zuopeng Fu\\Department of Mathematical Sciences\\University of Cincinnati\\2815 Commons Way\\Cincinnati, OH, 45221-0025, USA.}\email{fuzg@mail.uc.edu}
\author{Yizao Wang}\address{Yizao Wang\\Department of Mathematical Sciences\\University of Cincinnati\\2815 Commons Way\\Cincinnati, OH, 45221-0025, USA.}\email{yizao.wang@uc.edu}

\begin{document}\sloppy
\begin{abstract}
A new family of stable processes indexed by metric spaces with stationary increments are introduced. They are special cases of a new family of set-indexed stable processes with Chentsov representation. At the heart of the representation, a result on the so-called measure definite kernels is of independent interest. A limit theorem for set-indexed processes is also established. 
\end{abstract}

\date{\today}
\maketitle
\section{Introduction}
A stochastic process is most commonly referred to as a time-indexed collection of random variables. However, stochastic processes indexed by other generic sets, sometimes referred to as {\em parameterized} by metric spaces, also have a long history in probability theory. 
The probabilistic properties of stochastic processes are intrinsically connected to the geometry of the metric space, and the interactions can be very rich.  See for example \citep{adler90introduction,adler07random,narinucci11random}, just to mention a few.

Our motivating example is the Brownian motion. Paul L\'evy first introduced in the late 40s the notion of a Brownian motion indexed by a metric space, denoted by  $(\MM,\d)$ throughout.  When the Brownian motion, say $\{\B_x\}_{x\in \MM}$, exists, it is a centered Gaussian process determined by $\B_o = 0$ for a marked point $o\in \MM$ and 
\[
\esp(\B_x-\B_y)^2 = \d(x,y), \mfa x,y\in \MM.
\]
Under the assumption $\B_o =0$,  the above is equivalent to
\[
\cov(\B_x, \B_y) = \frac12\pp{\d(x,o) + \d(y,o) - \d(x,y)}. 
\]
Such a process is known as a {\em L\'evy Brownian field} \citep{levy65processus}. 
L\'evy first considered the case of $\MM = \Sn$ 
(the Euclidean sphere), then Chentsov \citep{chentsov57mouvement} addressed the case of $\MM = \Rn$ and then \citet{molchan67some} 
investigated the case of $\MM = \Hn$ 
(the hyperbolic space) 
and more general symmetric spaces. Since the so-defined processes are Gaussian, a necessary and sufficient condition for the existence is for the metric $\d$ to be of conditionally negative type \citep{gangolli67positive}. 
Nevertheless, corresponding {\em integral representations} of L\'evy Brownian fields (i.e.~stochastic integrals with respect to a white noise) have been developed too. In particular, \citet{takenaka81brownian} explained a general framework in terms of projective geometry that provides {\em Chentsov-type} integral representations for L\'evy Brownian fields indexed by $\Rn, \Sn$ and $\Hn$.
We shall focus on stochastic processes parameterized by these three metric spaces in this paper.  
In general, if the index manifold is not simply connected, then 
a L\'evy Brownian field does not exist  \citep{venet16existence}. 

There is a lately renewed interest of investigating fractional stochastic processes indexed by metric spaces. By fractional stochastic processes, we consider extensions of fractional Brownian motions, as stochastic processes indexed by $\R_+ = [0,\infty)$, to Gaussian processes indexed by generic metric spaces, and further to stable processes \citep{cohen12stationary,takenaka10stable,istas12manifold}. 
For extensions of L\'evy Brownian fields, as a natural extension of the fractional Brownian motions indexed by $\R_+$, we name a centered Gaussian process $\{\B^H_x\}_{x\in\MM}$ a {\em fractional L\'evy Brownian field}, as long as
\equh\label{eq:fBm}
\cov(\B^H_x,\B^H_y) = \frac12\pp{\d^\beta(x,o)+\d^\beta(y,o)-\d^\beta(x,y)} \qmwith H = \beta/2,
\eque 
and $\B^H_o = 0$ for some $o\in\MM$, or equivalently $\esp(\B_x^H-\B_y^H)^2 = \d^{\beta}(x,y)$. Note that the right-hand side of \eqref{eq:fBm} a priori is not a valid covariance function for all $\beta>0$. 
Such a framework has been recently considered by \citet{istas05spherical}, and the legitimate ranges for $\beta$ are known to be intervals in the form of $(0,\beta_\MM]$, with $\beta_{\Rn} = 2, \beta_{\Sn} = \beta_{\Hn} = 1$. 
It is worth noticing that for $\Sn$ and $\Hn$, the extension of fractional L\'evy Brownian fields only exists for $H\in(0,1/2]$.
Each fractional L\'evy Brownian field \eqref{eq:fBm} also has {\em stationary increments} with respect to 
a certain group action on $\MM$, see \eqref{eq:SI} below. 

Recent developments on fractional L\'evy Brownian fields include   for example \citep{lan18strong,cheng18extremes} (indexed by $\Sn$). 
More generally, in the spatial context a Gaussian process $\{G_x\}_{x\in\MM}$ is often characterized by its variogram $v(x,y):=\esp(G_x-G_y)^2$ (then stationary increments imply that $v(x,y)$ is a function of $\d(x,y)$), and there is already a huge literature on random fields from this aspect; see for example~\citep{bierme17introduction,xiao13recent} for latest surveys on Gaussian random fields. Other types of generalizations of fractional Brownian motions include for example \citep{herbin06set,molchanov15generalisation}.

Much fewer examples have been known for 
non-Gaussian
 stable processes indexed by metric spaces. We are in particular motivated by examples of such processes with stationary increments. 
A natural extension of \eqref{eq:fBm} to stable processes, say $\{Z_x\}_{x\in\MM}$, would require necessarily that
\equh\label{eq:istas}
\frac{Z_x-Z_y}{\d^{\beta/\alpha}(x,y)} \sim S_\alpha(\sigma,0,0), \mfa x,y\in \MM,
\eque
where the right-hand side stands for symmetric $\alpha$-stable (S$\alpha$S) distribution with scale parameter $\sigma>0$. The above relates the increments of the process and the geodesic distance in a unified manner as in the Gaussian case. 
However for $\alpha\in(0,2)$, unlike the Gaussian case $\alpha=2$, \eqref{eq:istas} is a strictly weaker notion than self-similarity and stationary increments even for $\MM = \R_+$, and is satisfied by different stable processes. 
Many examples of self-similar stable processes with stationary increments exist for $\MM = \R_+$ \citep{samorodnitsky94stable,pipiras17stable}. In contrast, for other choices of $\MM$, only the following examples are known to have stationary increments and satisfy \eqref{eq:istas}:
\begin{enumerate}[a.]
\item {\em L\'evy--Chentsov stable fields} \citep{takenaka10stable} as natural extensions of L\'evy Brownian fields sharing the same Chentsov representations (to be reviewed in  Section \ref{sec:LCsf}),
\item   {\em L\'evy--Chentsov sub-stable fields} (to be reviewed in Section \ref{sec:sub-stable}, revisited recently in \citep{istas06fractional}), and
\item {\em Takenaka stable fields} \citep{takenaka91integral} for $\MM = \Rn$ only (see also \citep[Chapter 8.4]{samorodnitsky94stable}).
\end{enumerate}

The main contributions of the paper are as follows.
\begin{enumerate}[1.]
\item A new family of stable processes, referred to as {\em fractional L\'evy--Chentsov stable fields}, are introduced, as an extension of fractional L\'evy Brownian fields.  These are S$\alpha$S processes indexed by $\MM = \Rn, \Sn, \Hn$, $\alpha\in(0,2]$, and are shown to have stationary increments and also satisfy \eqref{eq:istas} with $\beta\in(0,1)$ (Theorem \ref{thm:0}). They have Chentsov-type integral representation, which  for $\alpha = 2$ is new for fractional L\'evy Brownian fields. More generally, these processes are special cases of {\em set-indexed Karlin stable processes} that we shall introduce.

\item At the core of our presentation, a result on the so-called {\em measure definite kernels} \citep{robertson98negative}  is of its own interest from analysis point of view.  Recall that a metric $\d$ is a measure definite kernel, if there exists a measure space $(E,\calE,\mu)$ and a family of sets $A_x\in \calE$ for all $x\in \MM$, such that
\equh\label{eq:MDK}
\d(x,y) = \mu(A_x\Delta A_y), \mfa x,y\in \MM.
\eque
This property is strictly stronger than that $\d$ is of conditionally negative type, and has already been used in the Chentsov representation of L\'evy Brownian fields (see e.g.~\citep{istas12manifold,takenaka81brownian}). (This property
is also a special case of the Crofton formulae in integral geometry \citep{schneider08stochastic,robertson98crofton}.)
Here, it is shown that if $\d$ is a measure definite kernel, then for all $\beta\in(0,1)$, 
so is $\d^\beta$ with respect to a different measure space and sets $\{A_x^*\}_{x\in \MM}$ (Proposition \ref{Prop:1}). 
Based on this result, it follows immediately that fractional L\'evy--Chentsov stable fields satisfy \eqref{eq:istas} (Proposition \ref{prop:increments}).

\item A limit theorem is established for set-indexed Karlin stable processes (Theorem \ref{thm:1}), as a generalization of recent developments on the Karlin model \citep{durieu16infinite,durieu17infinite}, an infinite urn scheme originally considered by \citet{karlin67central} (see also \citep{gnedin07notes}).

\end{enumerate}

The paper is organized as follows. Section \ref{sec:review} reviews L\'evy--Chentsov stable fields and introduces the notations to be used later. Section \ref{sec:Karlin} introduces the set-indexed Karlin stable processes, and fractional L\'evy--Chentsov stable fields parameterized by 
metric spaces.
Section \ref{sec:MDK} explains the key result on measure definite kernels. 
Section \ref{sec:limit} 
establishes a limit theorem for the general set-indexed Karlin stable processes.

\section{L\'evy--Chentsov stable fields}\label{sec:review}

In this section, we review the notion of L\'evy--Chentsov stable fields parameterized by metric spaces \citep{takenaka91integral,takenaka10stable}. Most results have been known, and the goal here is to present a self-contained and systematic presentation in the framework of Chentsov random fields, which plays a crucial role in the following sections. 
\subsection{Chentsov representation}
Let $(\MM,\d)$ be a metric space. Later on we shall focus on $\MM = \Rn,\Sn, \Hn$, and $\d$ the corresponding geodesic metric. 
We take the convention throughout that an S$\alpha$S  distribution with scale parameter $\sigma>0$, denoted by $S_\alpha(\sigma,0,0)$, has characteristic function $\exp(-\sigma^\alpha|\theta|^\alpha)$ for $\alpha\in(0,2]$, including the Gaussian case $\alpha=2$.  
By a {\em L\'evy--Chentsov S$\alpha$S field indexed by $\MM$}, 
we consider the following stable process with Chentsov representation: 
\[
 Z_\alpha(x) := \int_E\ind_{A_x} dM_\alpha = M_\alpha(A_x), x\in\MM,
\]
 where $M_\alpha$ is an S$\alpha$S random measure on a measurable space $(E,\calE)$ with a $\sigma$-finite control measure $\mu$, for $\alpha\in(0,2]$, 
 and $\{A_x\}_{x\in\MM}$ is 
 a collection of elements from $\calE$ 
 such that $\mu(A_x)<\infty$. Recall that the random measure $M_\alpha$ evaluated at every measurable set $A\in \calE$, provided $\mu(A)<\infty$, is 
 distributed as $S_\alpha(\mu(A)^{1/\alpha},0,0)$,
 is $\sigma$-additive over 
  disjoint sets almost surely, and is independently scattered: that is, $M_\alpha$ over disjoint sets are independent. 
Moreover, the characteristic function of finite-dimensional distributions of $Z_\alpha$ is, 
 for all $d\in\N\equiv\{1,2,\dots\}, \vv\theta\in\Rd,x_1,\dots,x_d\in\MM$, $\Lambda^d = \{0,1\}^d\setminus\{(0,\dots,0)\}$,
\equh\label{eq:ch.f.}
\esp\exp\pp{i\summ j1d\theta_jZ_\alpha(x_j)} = \exp\pp{-\sum_{\vv\delta\in\Lambda^d}\abs{\aa{\vv\theta,\vv\delta}}^\alpha\mu\pp{\bigcap_{j=1}^dA_{x_j}^{\delta_j}}}, 
\eque
where $\aa{\vv\theta,\vv\delta} = \summ j1d \theta_j\delta_j$ and we follow the convention here and below\[
A^\delta = \begin{cases}
A & \delta = 1\\
A^c & \delta = 0
\end{cases}.
\]See \citep[Chapter 8.2]{samorodnitsky94stable} for more on Chentsov random fields indexed by $\MM = \Rn$. 

 By convention,  the L\'evy--Chentsov S$\alpha$S field is pinned down to zero at some point $o\in\MM$ ($Z_\alpha(o)=0$), so $\mu(A_o) = 0$. 
We are in particular interested in those processes that have stationary increments. This notion is well understood in the time series context. To introduce this notion for $\MM$-indexed processes, we consider in addition, a group action $G$  on the metric space $\MM$,  that is,  a mapping from $G\times \MM$ to $\MM$, denoted by $(g,x)\mapsto gx$. We then say that  $\{Z_\alpha(x)\}_{x\in \MM}$ {\em has stationary increments with respect to $G$}, if
\equh\label{eq:SI}
\ccbb{Z_\alpha(g(x))-Z_\alpha(g(o))}_{x\in \MM} \eqd \ccbb{Z_\alpha(x)}_{x\in\MM} \mfa g\in G.
\eque
In the case $\MM = \R$, take $G = \R$ and $gx = x+g$ for all $g,x\in\R$.
To verify \eqref{eq:SI}, for Gaussian fields, it suffices to verify the covariances. For non-Gaussian stable ones, the above condition can be checked by verifying
\equh\label{eq:SI0}
\mu\pp{\bigcap_{j=1}^d A_{g(x_j)}^{\delta_j}\Delta A_{g(o)}} = \mu\pp{\bigcap_{j=1}^dA_{x_j}^{\delta_j}},
\eque
for all $d\in\N, x_1,\dots,x_d\in \MM, \vv\delta \in\Lambda^d$
(see \citep[Theorem 8.2.6]{samorodnitsky94stable}). 

\begin{remark}The pinning-down assumption $Z_\alpha(o) = 0$ is only a convention. One can easily show that if 
$Z_\alpha$ has stationary increments 
in the sense of \eqref{eq:SI}, then for the process $Z_\alpha'$ defined by $Z_\alpha'(x) :=Z_\alpha(x) + Y, x\in\MM$ for any random variable $Y$, not necessarily depending on $Z_\alpha$, we have $\{Z_\alpha'(g(x))-Z_\alpha'(g(o))\}_{x\in\MM}\eqd \{Z_\alpha'(x)-Z_\alpha'(o)\}_{x\in\MM}$.
\end{remark}

\subsection{L\'evy--Chentsov stable fields}\label{sec:LCsf}
Now, we review the Chentsov representation of L\'evy--Chentsov S$\alpha$S fields indexed by a metric space $\MM = \Rn, \Sn, \Hn$. by specifying in each case the choice of
\equh\label{eq:notations}
(\MM,\d), o, G, (E,\calE,\mu), \{A_x\}_{x\in\MM}.
\eque
Here, for each choice of $(\MM,\d)$, $\d$ is the geodesic metric on $\MM$, $o$ is an (arbitrary) fixed starting point where the process is pinned down to zero, $G$ is a group action on $\MM$. We need the following assumption. 
\begin{assumption}\label{A:1}\begin{enumerate}[(i)]
\item \label{a:1}
  $(E,\calE,\mu)$ and $\{A_x\}_{x\in\MM}$ are chosen so that $\d$ is  a measure definite kernel associated with $(E,\calE,\mu), \{A_x\}_{x\in\MM}$ via \eqref{eq:MDK}. 
  \item \label{a:3} The group action $G$ acts also on $E$ as another group action, such that
\equh\label{eq:gA}
A_{g(x)}\Delta A_{g(y)} = g(A_x\Delta A_y) = (gA_x) \Delta (gA_y),\quad \mfa g\in G, x,y\in\MM.
\eque
  \item  \label{a:4} The measure $\mu$ is $G$-invariant.

  \item 
  $A_o = \emptyset$.

\end{enumerate}
\end{assumption}
\begin{remark}
 Assumptions \eqref{a:1} \eqref{a:3} and \eqref{a:4} imply that $g$  preserves the metric, since
\[
\d(g(x),g(y)) = \mu(A_{g(x)}\Delta A_{g(y)}) = \mu(g(A_x\Delta A_y)) = \mu(A_x\Delta A_y) = \d(x,y).
\]
\end{remark}
\begin{remark} Our presentation is slightly different from \citep{takenaka81brownian,takenaka10stable}, where the three cases can be put in a unified framework of projective geometry. See also with \citep{istas12manifold}.
\end{remark}
Below are the notations \eqref{eq:notations} in each case.
For Euclidean spaces and spheres, more background on group actions can be found in for example \citep{schneider08stochastic}. For hyperbolic spaces, see \citep{cohen12stationary} for a review for probabilists.

\begin{example}
[Euclidean space]
This is referred to as the L\'evy--Chentsov stable fields in \citep{samorodnitsky94stable}. We set $\MM = \Rn$, 
$E$ as the space of all hyperplanes of $\Rn$ not including the origin, parametrized as $
E: = \ccbb{(s, r): s \in \Sn, 0 < r < \infty}$, $A_x: = \ccbb{(s, r): s \in \Sn, 0 < r < \ip{s, x}}, x\in\Rn$ as
the set of all hyperplanes that separate $o$ and $x$, and $\mu(dsdr) = dsdr$ ($ds$ is the Lebesgue measure on $\Sn$). 
In this case, $G$ is the rigid body motion on $\Rn$, and acts on $E$ in the canonical way.
 \end{example}
\begin{example}[Euclidean sphere]
We set  $\MM = \Sn:= \{x\in\R^{n+1}: \nn x = 1\}$, 
and $E$ is the space of all totally geodesic submanifolds of  $\Sn$, each denoted by $h_x:= \ccbb{y\in\Sn:\aa{x,y} = 0}, x\in\Sn$. Set
$A_x := \{y\in\Sn:h_y \mbox{ separates $x$ and $o$}\}, x\in\Sn$,
$\mu$ is the Lebesgue measure on $\Sn$ and $G = SO(n+1)$.  The induced action on $E$ is $gh_x = h_{g(x)}$. 
\end{example}
\begin{remark}
An equivalent and more common representation is as follows:  introduce the hemisphere
$H_x:=\ccbb{y\in\Sn:\aa {x,y}>0}$, 
and define
\[
Z_\alpha'(x) := \int_E
(\ind_{H_x}-\ind_{H_o})
dM_\alpha, x\in\MM
\]
with the same S$\alpha$S random measure.  Indeed, one can show that $Z_\alpha'$ has the same distribution as
\[
\ccbb{\int_E \abs{\ind_{H_x}-\ind_{H_o}}dM_\alpha}_{x\in\MM} \eqd\ccbb{\int_E \ind_{H_x\Delta H_o}dM_\alpha}_{x\in\MM}
\]
by a straightforward calculation, and also $H_x\Delta H_o = A_x$ for all $x\in\MM$. 
\end{remark}

\begin{example}[Hyperbolic space]
For the sake of simplicity, we only describe $\H^2$. 
Consider $\MM = \D \equiv \{z\in\C:|z|<1\}\cong \H^2$, the Poincar\'e disc, with
\[
\d (z,z') = \frac12\log \frac{\abs{1-\wb zz'}+\abs{z-z'}}{\abs{1-\wb zz'}-\abs{z-z'}}, z,z'\in\D,
\]
and
\[
G \equiv SU(1,1) = \ccbb{
\pp{\begin{array}{rr}
\alpha & \beta\\
\wb\beta & \wb\alpha
\end{array}}: |\alpha|^2-|\beta|^2 = 1, \alpha,\beta\in\C}.
\]
  Let $E$ denote the collection of all geodesic lines of $\D$: each $h\in E$ is an intersection of an Euclidean circle, say $S_h$, with $\D$, including the diameters of $\D$ viewed as the intersections of Euclidean circles with infinite diameters. We parameterize $h\in E$ by $(\varphi,\psi)\in[0,\pi/2]\times[0,2\pi)$, with
$S_h = \sccbb{  e^{i(\psi+\varphi)},e^{i(\psi-\varphi)}}$, 
and take 
$\mu(dh) = (\sin\varphi)^{-2}cd\varphi d\psi$ (intuitively, $\varphi$ is for the size and $\psi$ for the direction). Then $G$ acts on $E$, $G^* = G$,  and $\mu$ is $SU(1,1)$-invariant on $E$. Take 
$A_z:=\ccbb{h\in E: h \mbox{ separates $z$ and $0$}}, z\in\D.$
\end{example}
\begin{remark} \label{rem:alpha=2}
Strictly speaking, in the case $\alpha = 2$, our representations above differ from the corresponding ones in the literature by a multiplicative constant of $\sqrt 2$. This is easily seen as for our random measure, $M_2(A)\sim S_2(\mu(A)^{1/2},0,0)$ is a centered Gaussian random variable with variance $2\mu(A)$, while often a Gaussian random measure with control measure $\mu$ evaluated at $A$ is defined to have variance $\mu(A)$. 
\end{remark}

\subsection{L\'evy--Chentsov sub-stable 
fields}\label{sec:sub-stable}
There is a simple trick to obtain a new S$\alpha$S field by multiplying to an old S$\alpha'$S one (with $\alpha'\in(\alpha,2]$) an independent totally skewed $\alpha/\alpha'$-stable random variable, and the so-obtained fields are known as {\em sub-stable fields} (or {\em sub-Gaussian} when $\alpha' = 2$). Fix $\alpha$ and $\alpha'$ so that $\alpha'\in(\alpha,2]$, and let $\xi$ be a totally skewed stable random variable with law
 determined by $\esp e^{-\theta\xi} = e^{-\theta^{\alpha/\alpha'}}$ for all $\theta\ge 0$. Then in our context, we refer to
\equh\label{eq:sub-stable}
Z_{\alpha,\alpha'}(x):= \xi^{1/\alpha'}Z_{\alpha'}(x), x\in\MM, 
\eque
as a {\em L\'evy--Chentsov sub-stable S$\alpha$S field}. 
 It is however not of Chentsov type. 
The characteristic function of $Z_{\alpha,\alpha'}$ is
\[
\esp\exp\pp{-i\summ j1d \theta_jZ_{\alpha,\alpha'}(x_j)} = \exp\pp{-\nn{\summ j1d \theta_j\ind_{A_{x_j}}}_{\alpha'}^{\alpha}}, \mfa \vv\theta\in\Rn,
\]
with $\nn \cdot_{\alpha'} = \int_E|\cdot |^{\alpha'}d\mu$. 
See \citep[Proposition 3.8.2]{samorodnitsky94stable} for more details. 
The fact that the right-hand side above is a valid characteristic function has been known since at least \citep{hardin81isometries}. Recently, \citet{istas06fractional} revisited this fact without making connection to the sub-stable representation \eqref{eq:sub-stable}. 
\section{A new family of stable processes}\label{sec:Karlin}
\subsection{Set-indexed Karlin stable processes}\label{sec:siKarlin}
We now introduce a family of set-indexed stable processes, of which our extensions of L\'evy Brownian fields are special cases. We shall understand the law of the processes by their finite-dimensional distributions (see Remark \ref{rem:version} for issues on their sample paths). 
Throughout, we assume that $\alpha\in(0,2]$ and $\beta\in(0,1)$.

Let $(E,\calE)$ be a measure space with a $\sigma$-finite measure $\mu$. Let $\calA$ be the family of subsets of $\calE$ with finite $\mu$-measure. 
We let $\mathbb M_p(E)$ denote the canonical space of Radon point measures on $(E,\calE)$, equipped with the Borel $\sigma$-algebra $\calM_p(E)$. In particular, every $m\in \mathbb M_p(E)$ takes the form $m = \sum_{i\in\N}\delta_{x_i}$ with $x_i\in E$, and for all $A\in\calE$, $m(A)\in\N\cup\{0,\infty\}$ and $m(A)<\infty$ if $A$ is compact in $\MM$. For background on topological issues, see \citep[Chapter 3]{resnick87extreme}.
Given $r>0$, let $\proba'_{r,\mu}$ denote the probability measure on $(\mathbb M_p(E),\calM_p(E))$ induced by the Poisson point process on $(E,\calE)$ with mean measure $r\cdot \mu$. That is, 
\[
\proba'_{r,\mu}(m(A) = k) = \frac{(r\mu(A))^k}{k!}e^{-r\mu(A)}, k = 0,1,\dots, A\in\calA.
\]
Set
\equh\label{eq:mu_beta}
 \mu_\beta(\cdot) := c_\beta \int_0^\infty r^{-\beta-1}\proba'_{r,\mu}(\cdot) dr  \qmwith c_\beta := \frac{\beta 2^{1-\beta}}{\Gamma(1-\beta)}
\eque
as a $\sigma$-finite measure on $(\mathbb M_p(E),\calM_p(E))$. We introduce the {\em set-indexed Karlin stable process}  given by
\equh\label{eq:Y}
Y_{\alpha,\beta}(A):= \int_{\mathbb M_p(E)}\inddd{m(A)\ \rm odd}\Mab(dm), A\in\calA,
\eque
where $\Mab$ is an S$\alpha$S random measure on $\mathbb M_p$ with control measure $\mu_\beta$. 
Note that the process $Y$ is still of Chentsov type: by introducing
\[
A^*:=\ccbb{m\in\mathbb M_p(E): m(A)\ \rm odd}, A\in\calA, 
\]
we have
\[
Y_{\alpha,\beta}(A) = \int_{\mathbb M_p(E)}\ind_{A^*}\Mab(dm).
\]
\begin{remark}
The original Karlin stable processes investigated in \citep{durieu17infinite} corresponds to 
\[
Y_{\alpha,\beta}(t)\equiv Y_{\alpha,\beta}([0,t]) = \int_{\mathbb M_p(\R_+)}\inddd{m([0,t]) \ \rm odd}\Mab(dm),
\] which has a more convenient representation
\equh\label{eq:Karlin}
\ccbb{Y_{\alpha,\beta}(t)}_{t\in\R_+} \eqd \ccbb{\int_{\R_+\times\Omega'}\inddd{N'(rt) \ \rm odd}\wt M_{\alpha,\beta}(drd\omega')}_{t\in\R_+},
\eque
where $N'$ is a standard Poisson point process on $\R_+$, defined on another probability space $(\Omega',\calF',\proba')$ and $\wt M_{\alpha,\beta}$ is an S$\alpha$S random measure on $\R_+\times\Omega'$ with control measure $c_\beta r^{-\beta-1}dr\proba'(d\omega')$. 
However, such a representation cannot be extended to the case $\MM = \Sn$ or $\Hn$ (e.g.~in the case of $\Sn$ the `scaled set' $rA_x$ would not make sense). Instead we work with stochastic integrals over $\MpE$.
\end{remark}
\begin{remark}Another representation of \eqref{eq:Y} with a flavor of {\em doubly stochastic processes} \citep{samorodnitsky94stable} is to write
\[
\ccbb{Y_{\alpha,\beta}(A)}_{t\in\calA} \eqd \ccbb{\int_{\R_+\times\Omega'}\inddd{N'_r(A)\ \rm odd}\wt M_{\alpha,\beta}'(drd\omega')}_{t\in\R_+},
\]
where $\{N_r'\}_{r>0}$ is a family of independent Poisson point processes on $(E,\calE)$ with intensity measure $r\mu$ respectively, defined on another probability space $(\Omega',\calF',\proba')$. 
\end{remark}
\begin{remark}\label{rem:version}
For the Karlin stable process in \eqref{eq:Karlin}, in the Gaussian case it is a fractional Brownian motion with Hurst index $H = \beta/2\in(0,1)$, and hence the path properties are well known. We expect then to be able to improve and obtain regularity results on sample paths (see \citep{lan18strong} for $\MM = \Sn$). In the stable case, however, even for $\MM = \R_+$ it remains open whether the Karlin stable process \eqref{eq:Karlin} has a version in the space  $D$ for $\alpha\in[1,2)$ (see \citep[Remark 2]{durieu17infinite}).
\end{remark}

\subsection{Fractional L\'evy--Chentsov stable fields}
For a metric space $(\MM,\d)$ along with notations \eqref{eq:notations} satisfying Assumption \ref{A:1},  we name the process
\equh\label{eq:eta}
\eab(x):= Y_{\alpha,\beta}(A_x), x\in \MM
\eque
as the {\em fractional L\'evy--Chentsov stable field} with parameters $\alpha\in(0,2], \beta\in(0,1)$. 
Our main result is the following.
\begin{theorem}\label{thm:0}
Each $\{\eab(x)\}_{x\in\MM}$ is an S$\alpha$S process  with stationary increments as in \eqref{eq:SI}.
\end{theorem}

\begin{proof}
It is equivalent to prove, for all $g\in G$, 
\[
\ccbb{\int_{\mathbb M_p(E)}\ind_{A_{g(x)}^*\Delta A_{g(o)}^*}\Mab(dm)}_{x\in\MM} \eqd \ccbb{\int_\MpE \ind_{A_x^*}\Mab(dm)}_{x\in\MM}.
\]
The above is equivalent to (recall \eqref{eq:SI0}), for all $d\in\N$, $\vv\delta\in\Lambda^d$, $x_1,\dots,x_d\in\MM$, 
\equh\label{eq:SI1}
\mu_\beta\pp{\bigcap_{j=1}^d(A_{g(x_j)}^*)^{\delta_j}\Delta A_{g(o)}^*} = \mu_\beta\pp{\bigcap_{j=1}^d(A_{x_j}^*)^{\delta_j}}.
\eque
We have
\begin{align*}
A_{g(x)}^*\Delta A_{g(y)}^* & = \ccbb{m(A_{g(x)}) \ \rm odd} \Delta \ccbb{m(A_{g(y)}) \ \rm odd}
 = \ccbb{m(A_{g(x)}\Delta A_{g(y)})\ \rm odd}\\  
& = \ccbb{m((gA_x)\Delta (gA_y))\ \rm odd} =  (gA_x)^*\Delta (gA_y)^*,  
\end{align*}
where in the third step we applied \eqref{eq:gA}, and
\[
(A_{g(x)}^*)^c\Delta A_{g(y)}^* = (A_{g(x)}^*\Delta A_{g(y)}^*)^c = \pp{(gA_x)^*\Delta (gA_y)^*}^c = (gA_x)^{*c}\Delta (gA_y)^*.
\]
Then, 
\eqref{eq:SI1} becomes (note $gA_o = \emptyset$)
\[
\mu_\beta\pp{\bigcap_{j=1}^d(gA_{x_j})^{*\delta_j}} = \mu_\beta\pp{\bigcap_{j=1}^d(A_{x_j}^*)^{\delta_j}}.
\]
The above shall follow from (recall $\mu_\beta$ in \eqref{eq:mu_beta}), for all $r>0$, 
\begin{multline}\label{eq:SI3}
\proba_{r,\mu}'\pp{m(gA_{x_j}) = \delta_j \mod 2, j=1,\dots,d}\\
= \proba_{r,\mu}'\pp{m(A_{x_j}) = \delta_j \mod 2, j=1,\dots,d}.
\end{multline}
The left-hand side above can be expressed as
\[
\wt \proba_{r,\mu}'\pp{m(A_{x_j}) = \delta_j \mod 2, j=1,\dots,d},
\]
where $\wt \proba_{r,\mu}'$ is the probability measure on $\MpE$ induced by the Poisson point process on $E$ with intensity measure $r\cdot \mu(g\cdot)$. 
 Since $\mu$ is $G$-invariant, the probability above is nothing but the right-hand side of \eqref{eq:SI3}. This completes the proof.
\end{proof}

The law of the increment of the field over any two points is uniquely determined by  their geodesic distance. 
\begin{proposition}\label{prop:increments}
For the stable process  $\eta_{\alpha,\beta}$ defined in \eqref{eq:eta}, 
we have that
\[
\frac{\eta_{\alpha,\beta}(x)-\eta_{\alpha,\beta}(y)}{\d^{\beta/\alpha}(x,y)} \sim S_\alpha(1,0,0), \mfa x,y\in \MM.
\]
\end{proposition}
\begin{proof}
This follows from a straightforward computation:
\begin{align*}
\esp\exp \pp{i\theta(\eab(x)-\eab(y))} 
&= \exp\pp{-|\theta|^\alpha\int_{\mathbb M_p(E)}{\inddd{m(A_x\Delta A_y)\ \rm odd} } \mu_\beta(dm)}\\
& = \exp\pp{-\mu_\beta(A_x^*\Delta A_y^*)|\theta|^\alpha} = \exp\pp{-\d^\beta (x,y)|\theta|^\alpha}. 
\end{align*}
In the last step, we applied an identity regarding $\d^\beta$ and $\mu_\beta$, established separately in Proposition \ref{Prop:1}.
\end{proof}
\begin{remark}The Karlin stable fields indexed by metric spaces
 $\eab\equiv \{\eta_{\alpha,\beta}(x)\}_{x\in\MM}$ are different from L\'evy--Chentsov stable fields and L\'evy--Chentsov sub-stable fields. 
To see this, it suffices to compare the spherical representations of finite-dimensional distributions \citep{samorodnitsky94stable}. The L\'evy--Chentsov sub-stable fields are spectrally continuous, while the other two are spectrally discrete, of which one could check readily that the spectral measures are different. 
In the case $\MM=\Rn$, it is also easily seen to be different from the spectrally discrete Takenaka random fields \citep{takenaka91integral}.
\end{remark}
We conclude this section with a few immediate consequences on the so-called set-indexed fractional Brownian motions investigated by \citet{herbin06set,herbin07multiparameter},  
whose motivation is different from ours.
\subsection{Set-indexed fractional Brownian motions}
Recall the set-indexed Karlin stable process $\{Y_{\alpha,\beta}(A)\}_{A\in\calA}$ as in \eqref{eq:Y}. In the Gaussian case, one could compute
\[
\cov(Y_{2,\beta}(A),Y_{2,\beta}(B)) = 
\mu^\beta(A)+\mu^\beta(B) - \mu^\beta(A\Delta B)
, A,B\in\calA,
\]
and recognize the covariance structure of set-indexed fractional Brownian motions (the above actually differs from the one in \citep{herbin06set,herbin07multiparameter} by a multiplicative constant of $1/2$; see a detailed calculation below).
A special choice of the index set is the collection of rectangles in Euclidean space 
\[
[\vv0,\vvt] \equiv \{\vvs =(s_1,\dots,s_n): 0\le s_j\le t_j, j=1,\dots,n\}, \vvt = (t_1,\dots,t_n)\in\R_+^n,
\]
in which case the process is referred to as 
a
 {\em multiparameter fractional Brownian motion}  (see also \citep{richard15fractional,richard17some}). This process has different properties from other extensions of fractional Brownian motions, and in particular it does not have stationary increments (see e.g.~\citep{richard15fractional} for comparisons). 
Following the same notion, we refer to 
\[
\xab(\vvt):= Y_{\alpha,\beta}(\vvt)\equiv \int_{\MpE}\inddd{m([\vv0,\vvt])\  \rm odd}\Mab(dm), \vvt\in\R_+^n
\]
as 
a {\em multiparameter fractional stable field}. This integral representation in the Gaussian case ($\alpha = 2$) seems new. 

We next consider a natural decomposition of set-indexed fractional Brownian motions $Y_{2,\beta}$, inspired by the decomposition of fractional Brownian motions by bi-fractional Brownian motions introduced by \citet{lei09decomposition}.  Write $W_\beta \equiv Y_{2,\beta}$ from now on. We consider a slightly different representation of $W_\beta$:
\[
\ccbb{W_\beta(A)}_{A\in\calA} \eqd \ccbb{\int_{\R_+\times \MpE} \inddd{m(A) \ \rm odd}\wt M_{2,\beta}(drdm)}_{A\in\calA},
\]
where $\wt M_{2,\beta}$ is a Gaussian (S$\alpha$S with $\alpha=2$) random measure  on $\R_+\times\MpE$ with control measure $c_\beta r^{-\beta-1}dr\proba'_{r,\mu}(dm)$,  
and then consider its decomposition
\[
\ccbb{W_\beta(A)}_{A\in\calA}  \eqd \ccbb{W_{1,\beta}(A)+W_{2,\beta}(A)}_{A\in\calA}
\]
with
\begin{align*}
W_{1, \beta}(A) & := \int_{\R_+\times \MpE} \pp{\inddd{m(A) \ \rm odd} - q_r(A)}\wt M_{2,\beta}(drdm)\\
W_{2, \beta}(A) & := \int_{\R_+\times \MpE} q_r(A) \wt M_{2,\beta}(drdm)\\
q_r(A)& := \proba_{r,\mu}'\pp{m(A) \ \rm odd}.
\end{align*}
For the case $\MM = \R_+$ and relation to decomposition of fractional Brownian motions, see \citep[Section 2.2]{durieu16infinite}. 
\begin{proposition} 
$\{W_{1, \beta}(A)\}_{A\in\calA}$ and $\{W_{2, \beta}(A)\}_{A\in\calA}$ are independent centered Gaussian processes with covariance functions, for $A_1, A_2 \in \calA$,
\begin{align*}
\cov(W_{1, \beta}(A_1), W_{1, \beta}(A_2)) & = 
(\mu(A_1) + \mu(A_2))^{\beta} - \mu^{\beta}(A_1\Delta A_2)
\\
\cov(W_{2, \beta}(A_1), W_{2, \beta}(A_2)) & = 
\mu^{\beta}(A_1) + \mu^{\beta}(A_2) - (\mu(A_1) + \mu(A_2))^{\beta}.
\end{align*}
\end{proposition}

\begin{proof}
Recall that for a Poisson random variable $Z$ with parameter $\lambda>0$, 
\equh\label{eq:odd}
\proba(Z \mbox{ is odd}) = \frac12(1-e^{-2\lambda}).
\eque We first compute the covariance function of $\{W_{2, \beta}(A)\}_{A\in\calA}$. 
For $A_1, A_2 \in\calA$,
\begin{align*}
\cov(W_{2, \beta}(A_1), W_{2, \beta}(A_2)) & = 
2
\int_{\R_+} q_r(A_1)q_r(A_2) c_{\beta}r^{-\beta-1}dr \\
& = 2\int_{\R_+} \frac{1}{4} \pp{1-e^{-2r\mu(A_1)}} \pp{1-e^{-2r\mu(A_2)}}c_{\beta}r^{-\beta-1}dr 
\end{align*}
(the factor 2 is due to our convention for the Gaussian random measure, see Remark \ref{rem:alpha=2}),
and for $\beta\in(0,1)$, the desired formula then follows immediately from the identity for the Gamma function
\equh \label{gammafunction}
\int_0^\infty (1-e^{-ar})r^{-\beta-1}dr = \frac{a^\beta\Gamma(1-\beta)}{\beta}, a>0.
\eque
Similarly, 
\begin{multline*}
 \cov(W_{1, \beta}(A_1), W_{1, \beta}(A_2)) \\
 = 2\int_{\R_+\times \MpE} \pp{\proba'_{r,\mu}\pp{m(A_1) \ {\rm odd}, m(A_2) \ \rm odd} - q_r(A_1)q_r(A_2)} c_{\beta}r^{-\beta-1}dr,
\end{multline*}
and  
\[
\proba'_{r,\mu}\pp{m(A_1) \ {\rm odd}, m(A_2) \ \rm odd} = \frac14\pp{1-e^{-2r\mu(A_1)}-e^{-2r\mu(A_2)}+e^{-2r\mu(A_1\Delta A_2)}}.
\] The desired covariance formula  then follows. 
The independence of $\{W_{1, \beta}(A)\}_{A\in\calA}$ and $\{W_{2, \beta}(A)\}_{A\in\calA}$ can be verified similarly. 
\end{proof}
\section{Measure definite kernels}\label{sec:MDK}
In this section, we extract a result on measure definite kernels that we developed and used implicitly in our previous analysis. This result is of independent interest for metric analysis.

Let $(\MM,\d)$ be a metric space. Recall that the metric $\d$ is of {\em conditionally negative type}, if for all $n\ge 2$ and $x_1,\dots,x_n\in \MM$ and $\lambda_1,\dots,\lambda_n\in\R$ with $\summ i1n \lambda_i = 0$, we have
$\summ i1n\summ j1n \lambda_i\lambda_j\d(x_i,x_j)\le 0$.
Recall the definition of a metric $\d$ being a measure definite kernel in \eqref{eq:MDK}. Sometimes it is convenient to write equivalently
\[
\d(x,y) = \int \abs{\ind_{A_x}-\ind_{A_y}} d\mu, \mfa  x,y\in \MM.
\]
A measure definite kernel as a metric  is necessarily of conditionally negative type, as for all collections $\{A_i\}_{i=1,\dots,n}\subset \{A_x\}_{x\in \MM}$,
\[
\pp{\summ i1n \lambda_i\ind_{A_i}}^2  = \frac12\summ i1n \summ j1n \lambda_i\lambda_j\pp{\ind_{A_i}+\ind_{A_j}-\ind_{A_i\Delta A_j}} = -\frac12\summ i1n\summ j1n \lambda_i\lambda_j \ind_{A_i\Delta A_j}.
\] Although the converse is not true.
It is well known that the mapping $\d \mapsto \d^\beta$ preserves the property of being conditionally negative. The following shows that the mapping also preserves the property of being a measure definite kernel. 
Recall notations for $\mathbb M_p(E), \calM_p(E),\mu_\beta,\proba_{r,\mu}',A^*$ as in Section \ref{sec:siKarlin}.

\begin{proposition}\label{Prop:1}
Suppose that $(M,\d)$ is a metric space and  $\d$ is a measure definite kernel with respect to a measure space $(E,\calE,\mu)$ and $\{A_x\}_{x\in \MM}\subset \calE$, as in 
\eqref{eq:MDK}. Then for all $\beta\in(0,1)$, 
\[
\d^\beta(x,y)
 =\mu_\beta(A_x^*\Delta A_y^*),
\]
with $\mu_\beta$ as in \eqref{eq:mu_beta} and $A^* := \ccbb{m\in\MpE: m(A) \ \rm odd}, A\in\calA$. 
\end{proposition}
\begin{proof}
First, since $\mu(A_x\Delta A_y) = \d(x,y)$, by \eqref{gammafunction}
we write 
\[
\d^\beta(x,y) = \int_{\R_+}\frac12\pp{1-e^{-2r\mu(A_x\Delta A_y)}}c_\beta r^{-\beta-1}dr.
\]
On the other hand, for every $r>0$, recalling \eqref{eq:odd},  we have
\begin{align*}
\d^\beta(x,y)
& = \int_{\R_+}\proba'_{r,\mu}(\{m(A_x)\ {\rm odd}\}\Delta \{m(A_y)\ {\rm odd}\})c_\beta r^{-\beta-1}dr\\
& =\int_{\R_+} \int_{\mathbb M_p(E)}\abs{\inddd{m(A_x)\ \rm odd} - \inddd{m(A_y)\ \rm odd}} \proba'_{r,\mu}(dm)c_\beta r^{-\beta-1}dr\\
& = \int_{\mathbb M_p(E)}\abs{\ind_{A_x^*} - \ind_{A_y^*}} \mu_\beta(dm) = \mu_\beta(A_x^*\Delta A_y^*).
\end{align*}

\end{proof}
\section{A limit theorem for set-indexed Karlin stable processes}\label{sec:limit}
Limit theorems for stochastic processes indexed by 
metric spaces
 are not new (e.g.~\citep{lavancier07invariance,ossiander85levy,bierme10selfsimilar,bierme14invariance}). However, very few theorems have been known beside $\R^n$ and $\R^n_+$-indexed examples, with the notable exception for sphere-indexed processes investigated by \citet{estrade10ball}.
Our model is a variation of an infinite urn scheme considered by \citet{karlin67central}. 
This version of the model was introduced for $\MM = \R_+$ in the proofs of \citep{durieu17infinite} as the Poissonized version of the corresponding Karlin model, and has an aggregation nature.

Let $(E, \calE, \mu)$ be a measure space with $\mu$ a $\sigma$-finite measure. 
Let $\{p_k\}_{k\in\N}$ be prescribed strictly positive numbers, and $\rho>0$ a scaling parameter that eventually goes to infinity. 
For each $\rho>0$, let $\{N_k\topp \rho\}_{k \in \N}$ be a family of independent Poisson point processes with mean measure $\rho p_k \cdot \mu$ on $(E, \calE)$ and let $\{\varepsilon_k\}_{k \in \N}$ be another family of i.i.d.~random variables, independent from the Poisson point processes.  
We shall restrict to those sets in $\calE$ with finite $\mu$-measure, the collection of which denoted by $\calA$. Our model is then defined as
\[
U_\rho(A) := \sum_{k \in \N} \varepsilon_k \inddd{N_k\topp\rho(A) \ \mbox{odd}}, \rho>0, A \in \calA.
\]
We assume, with
$\nu := \sum_{k \in \N} \delta_{1/p_k}$, that for some $\beta \in (0, 1)$, 
\begin{equation} \label{eq:1}
\nu(t) = \max\{k \in\N: p_k \geq 1/t\} = t^\beta L(t), \quad t >0,
\end{equation}
where $L$ is a slowly varying function at $\infty$, and 
 the characteristic function $\phi(\theta) = \esp\exp(i\theta\varepsilon_1)$ satisfies
\equh\label{eq:ch.f.epsilon}
1-\phi(\theta) \sim \sigma_\varepsilon^\alpha|\theta|^\alpha \mmas \theta\to 0.
\eque
The above follows for example if for some $\alpha \in (0, 2)$,
\[
\lim_{x \to \infty} \frac{\P (|\varepsilon_1| > x)}{x^{-\alpha}} = C_\varepsilon \in (0, \infty),
\]
and $\sigma_{\varepsilon}^{\alpha} = C_\varepsilon \int_{0}^{\infty} x^{-\alpha}\sin x dx$
(e.g.~\citep[Theorem 8.1.10]{bingham87regular}),
or if $\alpha = 2$, $\varepsilon_1$ is centered and 
$\sigma_{\varepsilon}^2 = \esp\varepsilon_1^2/2<\infty$.  Recall $c_\beta$ in \eqref{eq:mu_beta}.
\begin{theorem} \label{thm:1}
Under assumptions \eqref{eq:1} and \eqref{eq:ch.f.epsilon}, 
\[
\left(\frac{U_\rho(A)}{b_\rho}\right)_{A \in \calA} \fddto  
\left(Y_{\alpha,\beta}(A)\right)_{A \in \calA} \qmwith b_\rho:=\pp{\frac{\beta}{c_\beta}\rho^\beta L(\rho)}^{1/\alpha}\sigma_\varepsilon.
\]\end{theorem}
\begin{proof}
Consider $d\in\N$, 
$\vvA=(A_1, \dots, A_d) \in \calA^d$ and $\boldsymbol{\delta}=(\delta_1, \dots, \delta_d) \in \Lambda_d$, and for any point measure $m$ on $(E,\calE)$, for the sake of simplicity write
\[
\ccbb{\vv m(\vvA)=\boldsymbol{\delta} \mmod 2} \equiv \ccbb{m(A_j)=\delta_j \mmod 2 \ \mbox{for all} \ j=1, \dots, d}.
\]
The following statistics play a crucial role:
\[M_\rho^{\boldsymbol{\delta}}(\vvA) := \sum_{k \in \N} \inddd{\vvN_k\topp\rho(\vvA)=\vv\delta \mmod 2},
\]
Let $N\topp r$ be a Poisson point process on $(E,\calE)$ with intensity measure $r\mu(\cdot)$. 
Introduce
\begin{align*}
\mathfrak m^{\vv \delta}(\vvA) &:= \int_{\R_+}\proba\pp{\vv N\topp r(\vv A) = \vv \delta \mmod 2}c_\beta r^{-\beta-1}dr\\
& = \int_{\MpE}\inddd{\vv m(\vv A) = \vv \delta \mmod 2}\mu_\beta(dm).
\end{align*}
The key estimate is
\equh\label{eq:limit_Mn}
\lim_{\rho \to \infty} \frac{M_\rho^{\boldsymbol{\delta}}(\vvA)}{b_\rho^\alpha} = \frac{\mathfrak m^{\boldsymbol{\delta}}(\vvA)}{\sigma_\varepsilon^\alpha}
 \mbox{ in probability.}
\eque
Observe that 
\begin{align*}
\E M_\rho^{\boldsymbol{\delta}}(\vvA) 
&= \sum_{k \in \N} \P \pp{\vvN_k\topp\rho(\vvA) = \boldsymbol{\delta} \mmod 2} 
 = \int_{0}^{\infty} \P \pp{\vvN^{(\rho/x)}(\vvA)=\boldsymbol{\delta} \mmod 2} \nu (dx).
\end{align*}
Set $\varphi_{\vvA}(s):=\P (\vvN^{(s)}(\vvA)=\boldsymbol{\delta} \mmod 2)$. It is easy to see that $\varphi_{\vvA}$ is differentiable and vanishes at zero. Therefore, integrating by parts, we have
$$\E M_\rho^{\boldsymbol{\delta}}(\vvA) = \int_{0}^{\infty}\varphi_{\vvA}(\rho/x)\nu(dx) = \int_{0}^{\infty} \frac{1}{x^2}\varphi_{\vvA}'\pp{\frac{1}{x}}\nu(\rho x)dx.$$
Then,
\begin{align*}
\lim_{\rho \to \infty} \frac{\E M_\rho^{\boldsymbol{\delta}}(\vvA)}{\rho^{\beta}L(\rho)} & = \lim_{\rho \to \infty} \int_{0}^{\infty} \frac{1}{x^2}\varphi_{\vvA}'\pp{\frac{1}{x}} \frac{\nu(\rho x)}{\rho^{\beta}L(\rho)}dx \\
& = \int_{0}^{\infty} \varphi_{\vvA}'\pp{\frac{1}{x}}x^{\beta-2}dx = \int_{0}^{\infty}\varphi_{\vvA}\pp{\frac{1}{x}} \beta x^{\beta-1}dx\\
& = \int_{0}^{\infty} \P\pp{\vvN^{(r)}(\vvA)=\boldsymbol{\delta} \mmod 2} \beta r^{-\beta-1}dr,
\end{align*}
where in the second step, we applied the assumption \eqref{eq:1} on $\nu$, and the interchange of the limit and the integral can be verified as explained in \citep[Lemma 1]{durieu17infinite}.
Furthermore, since 
$\var \spp{M_\rho^{\boldsymbol{\delta}}(\vvA)} 
 \leq \sum_{k \in \N} \P \spp{\vvN^{(\rho)}_k(\vvA)=\boldsymbol{\delta} \mmod 2} = \E M_\rho^{\boldsymbol{\delta}}(\vvA)$,
the $L^2$ convergence holds. We have thus shown \eqref{eq:limit_Mn}.

Now we prove the desired convergence by computing the characteristic function of the finite-dimensional distribution. 
For all $\vv\theta\in\R^d$, we write
\begin{align*}
\E \exp\pp{i\sum_{j=1}^{d} \theta_j \frac{U_\rho(A_j)}{b_\rho}} &= \E \exp\pp{i\sum_{j=1}^{d} \frac{\theta_j}{b_\rho}\sum_{k \in \N} \varepsilon_k \inddd{N_k\topp\rho(A_j) \ \mbox{odd}}} \\
& = \E \exp\pp{i\sum_{\boldsymbol{\delta} \in \Lambda_d} \frac{\abs{\ip{\boldsymbol{\theta}, \boldsymbol{\delta}}}}{b_\rho} \sum_{k \in \N} \varepsilon_k \inddd{\vv N_k\topp\rho(\vv A) = \vv\delta \mmod 2}}.
\end{align*}
Let $\calN$ denote the $\sigma$-algebra generated by $\{N_k\}_{k \in \N}$. 
Recall that $\phi$ is the characteristic function of $\varepsilon_1$.
Then, the above expression becomes
$$\E \ccbb{\prod_{\boldsymbol{\delta} \in \Lambda_d} \E \bb{\exp \pp{i \frac{\abs{\ip{\boldsymbol{\theta}, \boldsymbol{\delta}}}}{b_\rho}\sum_{\ell=1}^{M_\rho^{\boldsymbol{\delta}}(\vvA)}\varepsilon_\ell^{\boldsymbol{\delta}}} \mmid \calN}} = \E \pp{\prod_{\boldsymbol{\delta} \in \Lambda_d} \phi\pp{\frac{\abs{\ip{\boldsymbol{\theta}, \boldsymbol{\delta}}}}{b_\rho}}^{M_\rho^{\boldsymbol{\delta}}(\vvA)}}. $$
Therefore
\begin{align*} 
\E \exp\pp{i\sum_{j=1}^{d} \theta_j \frac{U_\rho(A_j)}{b_\rho}} & = \E \exp \pp{\sum_{\boldsymbol{\delta} \in \Lambda_d} M_\rho^{\boldsymbol{\delta}}(\vvA) \log \phi \pp{\frac{\abs{\ip{\boldsymbol{\theta}, \boldsymbol{\delta}}}}{b_\rho}}} \\ 
& = \E \exp \pp{-\sum_{\boldsymbol{\delta} \in \Lambda_d} \frac{M_\rho^{\boldsymbol{\delta}}(\vvA)}{b_\rho^{\alpha}} \abs{\ip{\boldsymbol{\theta}, \boldsymbol{\delta}}}^{\alpha} \sigma_{\varepsilon}^{\alpha} \frac{\log \phi \pp{\frac{\abs{\ip{\boldsymbol{\theta}, \boldsymbol{\delta}}}}{b_\rho}}}{-\sigma_{\varepsilon}^{\alpha} \abs{\frac{\ip{\boldsymbol{\theta}, \boldsymbol{\delta}}}{b_\rho}}^{\alpha}}}.
\end{align*}
Recall that the assumption \eqref{eq:ch.f.epsilon} implies that
$\log \phi\pp{\theta} \sim \phi\pp{\theta}-1 \sim -\sigma_{\varepsilon}^{\alpha} \abs{\theta}^{\alpha} \ \mbox{as} \ \theta \to 0$.
Then, by \eqref{eq:limit_Mn} and the dominated convergence theorem,
\[
\lim_{\rho\to\infty}\E \exp\pp{i\sum_{j=1}^{d} \theta_j \frac{U_\rho(A_j)}{b_\rho}} = \exp \pp{-\sum_{\boldsymbol{\delta} \in \Lambda_d} \abs{\ip{\boldsymbol{\theta}, \boldsymbol{\delta}}}^{\alpha} \mathfrak m^{\boldsymbol{\delta}}(\vvA)},
\]
where the limit is the desired characteristic function $\E \exp  \spp{i \sum_{j=1}^{d} \theta_jY_{\alpha, \beta}(A_j)}$ (recall \eqref{eq:ch.f.}). This completes the proof.
\end{proof}

\subsection*{Acknowledgement}
The authors would like to thank an anonymous referee for careful reading of the paper. 
YW thanks Olivier Durieu and Ilya Molchanov for careful reading and inspiring discussions on an earlier draft of the paper.
ZF and YW's research were partially supported by Army Research Laboratory grant W911NF-17-1-0006. 
YW's research was in addition partially supported by NSA grant H98230-16-1-0322.

\bibliographystyle{apalike}
\bibliography{references,references18}

\end{document}